\documentclass[11pt]{amsart}
\usepackage{amssymb,latexsym,cite}
\usepackage{pstcol,pstricks,color}
\usepackage{amsfonts, amsmath}
\usepackage{graphicx}
\usepackage{comment}
\usepackage{tikz}
 \usepackage{pstricks}

\advance\textheight by \topskip   \numberwithin{equation}{section}

\newtheorem{theorem}{Theorem}

\newtheorem{corollary}[theorem]{Corollary}

\newtheorem{lemma}[theorem]{Lemma}

\def\qq{{\mathbf q} }
\def\pp{{\mathbf p} }
\def\BB{\mathcal{B}}
\def\cora{\text{cor}_{\text{A}}}
\def\corb{\text{cor}_{\text{B}}}

\begin{document}

\begin{center}

{\large\bf  ENUMERATIONS OF BARGRAPHS WITH RESPECT TO CORNER STATISTICS
\rule{0mm}{6mm}\renewcommand{\thefootnote}{}
\footnotetext{

\rule{2.4mm}{0mm}Keywords and Phrases. Bargraphs; Corners; Set partitions; Stirling numbers; Bell numbers.
}}

\vspace{1cc}
{\large\it Toufik Mansour and G\"{o}khan Y\i ld\i r\i m}

\vspace{1cc}
\parbox{24cc}{{\small

We study the enumeration of bargraphs with respect to some corner statistics. We find generating functions for the number of bargraphs that track the corner statistics of interest, the number of cells, and the number of columns. We also consider bargraph representation of set partitions and obtain some explicit formulas for the number of specific types of corners in such representations.

}}
\end{center}

\vspace{1cc}

\vspace{1.5cc}
\begin{center}
{\bf 1. Introduction}
\end{center}

Combinatorial analysis of certain geometric cluster models such as polygons, polycubes, polyominos is an important research endeavor for understanding many statistical physics models \cite{Fer, G, JR}. A finite connected union of unit squares on two dimensional integer lattice is called a \textit{polyomino}, and a \textit{bargraph} is a column-convex polyomino in the first quadrant of the lattice such that its lower boundary lies on the $x$-axis.
A \textit{bargraph} can also be considered as a self-avoiding path in the integer lattice
$\mathbb{L}=\mathbb{Z}_{\geq 0}\times \mathbb{Z}_{\geq 0} $ with steps $u=(0,1)$, $h=(1,0)$ and $d=(0,-1)$ that
starts at the origin, ends on the $x$-axis and never touches the
$x$-axis except at the endpoints. The steps $u,h$ and $d$ are
called \textit{up, horizontal} and \textit{down} steps
respectively. Enumerations of bargraphs with respect to some statistics have been an active area of research recently \cite{Fer,Man,PB}. Bosquet-Mel\'ou and Rechnitzer \cite{BMR} obtain the site-perimeter generating function for bargraphs, and also show that it is not D-finite. Blecher et al. investigated the generating functions for bargraphs with respect to some statistics such as the number of levels \cite{Bl1}, descents \cite{Bl2}, peaks \cite{Bl3}, and walls \cite{Bl4}. Deutsch and Elizalde \cite{DE} used a bijection between bargraphs and cornerless Motzkin paths, and determined more than twenty generating functions for bargraphs according to the number of up steps, the number of horizontal steps, and the statistics of interest such as the number of double rises and double falls, the length of the first descent, the least column height. Bargraphs are also used in statistical
physics to model vesicles or polymers \cite{SP, AP, PB}.

We shall study the enumerations of bargraphs and set partitions with respect to some corner statistics. We shall first introduce some definitions. A unit square in the lattice $\mathbb{L}$ is
called a \textit{cell}. We identify a bargraph with a
sequence of numbers $\pi=\pi_1\pi_2\cdots \pi_m$ where $m$ is the number
of horizontal steps of the bargraph and $\pi_j$ is the number of
\textit{cells} beneath the $j^{th}$ horizontal step which is also
called the \textit{height} of the $j^{th}$ \textit{column}. A vertex on a bargraph is called a \textit{corner} if it is at the
intersection of two different types of steps. A corner is called an \textit{$(a,b)$-corner} if it is formed by maximum number $a$ of one type of  consecutive steps followed by maximum number $b$ of another type of consecutive steps.
A corner is called of \textit{type A} if it is
formed by down steps followed by horizontal steps $(\llcorner)$. Similarly, a corner is of \textit{type B} if it is formed by horizontal steps followed by down steps $(\urcorner)$, see Figure~\ref{barFig}.  We use $\BB_n$ and $\BB_{n,k}$ to denote the set of all bargraphs with $n$ cells, and the set of all bargraphs with $n$ cells and $k$ columns respectively.

Bargraphs are also related to the set partitions. Recall that a
\text{partition} of set $[n]:=\{1,2,\cdots,n\}$ is any collection
of nonempty, pairwise disjoint subsets whose union is $[n]$. Each
subset in a partition is called a \textit{block} of the partition.
A partition $p$ of $[n]$ with $k$ blocks is said to be in the
\textit{standard form} if it is written as $p=A_1/A_2/\cdots/A_k$
where $\min(A_1)<\min(A_2)<\cdots<\min(A_k)$. There is also a
unique \textit{canonical sequential representation} of a partition
$p$ as a word of length $n$ over the alphabet $[k]$ denoted by $\pi=\pi_1\pi_2\cdots \pi_n$ where $\pi_i=j$ if
$i\in A_{\pi_j}$ which can be considered a bargraph
representation. For instance, the partition
$\pi=\{1,3,6\}/\{2,5\}/\{4,7\}/\{8\}$ has the canonical sequential
representation $\pi=12132134$. Mansour \cite{Man} studied the generating functions for the number of set partitions of $[n]$ represented as bargraphs according to the number of interior vertices. For some other enumeration results, see also \cite{Bl5, MShaSha}. Henceforth, we shall represent set
partitions as bargraphs corresponding to their canonical sequential representations.
\begin{figure}
\centering \setlength{\unitlength}{1.3mm}
\begin{picture}(30,20)
\linethickness{0.4mm} \put(0,0){\line(0,1){6.8}}
\put(0,6.8){\line(1,0){3.4}} \put(3.4,6.8){\line(0,1){6.8}}
\put(3.4,13.6){\line(1,0){10.2}} \put(13.6,
13.6){\line(0,-1){10.2}} \put(13.6,3.4){\line(1,0){6.8}}
\put(20.4,3.4){\line(0,1){6.8}} \put(20.4,10.2){\line(1,0){3.4}}
\put(23.8,10.2){\line(0,-1){3.4}} \put(23.8,6.8){\line(1,0){6.8}}
\put(30.6,6.8){\line(0,-1){6.8}} \put(13.8,13.9){$ a$}
\put(12.6,2.0){$b$} \put(23.9,10.5){$ c$} \put(22.8,4.8){$d$}
\linethickness{0.075mm} \multiput(0,0)(3.4,0){11} {\line(0,1){17}}
\multiput(0,0)(0,3.4){6} {\line(1,0){34}}
\end{picture}
\caption{The bargraph $\pi=244411322$.  Type A corners $b$ and $d$ are $(3,2)$ and $(1,2)$-corners respectively. Type B corners $a$ and $c$ are $(3,3)$ and $(1,1)$-corners respectively.}
\label{barFig}
\end{figure}
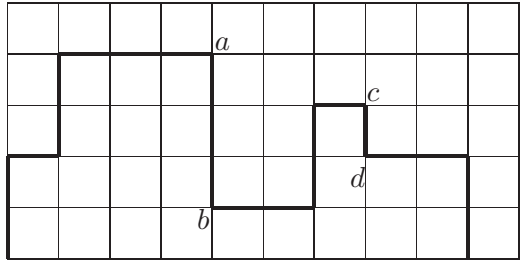

The rest of the paper is organized as follows. In section~2, we find the generating function for the
number of bargraphs according to the number of cells, the number
of columns, and the number of $(a,b)$-corners of type A for any given positive integers $a,b$. As a
corollary, we determine the total number of $(a,b)$-corners of type A, and the total number of type A corners over all bargraphs having $n$ cells. In section~2.3 and section~2.4, we extend these results to the restricted bargraphs in which the height of each column is restricted to be a maximum of $N$ for any given positive integer $N$, and to the set partitions respectively. We obtain
similar results for corners of type B in section~3.
One of the main results of the paper, Theorem~\ref{thmcAsp}, shows that the total number of corners of type A over the set partitions of $[n+1]$ with $k$ blocks is given by
$$\frac{n}{2}S_{n+1,k}-\frac{1}{4}S_{n+2,k}-\frac{n}{2}S_{n,k}+\frac{1}{4}S_{n+1,k}+\frac{1}{4}S_{n,k-2},$$
where $S_{n,k}$ is the Stirling number of second kind.
Similarly, Theorem~\ref{thmcBsp}, shows that the total number of corners of type B over the set partitions of $[n+1]$ with $k$ blocks is given by
$$\frac{n}{2}S_{n+1,k}-\frac{1}{4}S_{n+2,k}-\frac{n}{2}S_{n,k}+\frac{5}{4}S_{n+1,k}+\frac{1}{4}S_{n,k-2}.$$

\vspace{1.5cc}
\begin{center}
{\bf 2. Counting Corners of type A}
\end{center}

Let $H:=H(x,y,\qq)$ be the generating function for the number of bargraphs $\pi$ according to the number of cells
in $\pi$, the number of columns of $\pi$, and the number of
$(a,b)$-corners of type A in $\pi$ corresponding to the variables
$x,y$ and $\qq=(q_{a,b})_{a,b\geq 1}$ respectively. That is,
$$H=\sum_{n\geq 0}\sum_{\pi\in \BB_n}x^ny^{\text{col}(\pi)}\prod_{a,b\geq
1}q_{a,b}^{\Lambda_{(a,b)}(\pi)},$$
where $\Lambda_{(a,b)}(\pi)$ is the number of $(a,b)$-corners of
type A in $\pi$, and $\text{col}(\pi)$ denotes the number of columns of $\pi$.

From the definitions, we have
\begin{equation}\label{eqcA1}
H=1+\sum_{a\geq 1}H_a,
\end{equation}
where $1$ counts the empty bargraph, and $H_{a_1a_2\cdots a_s}:=H_{a_1a_2\cdots a_s}(x,y,\qq)$ is the generating function
for the number of bargraphs $\pi=a_1a_2\cdots a_s\pi'$ in which the height of the $j^{th}$ column is $a_j$, where $j=1,2,\ldots,s$. Since each bargraph
$\pi=a\pi'$ can be
decomposed as either $a$, $aj\pi''$ with $j\geq a$ or $ab\pi''$
with $1\leq b\leq a-1$, we have
\begin{equation}\label{eqcA2}
H_a=x^ay+x^ay\sum_{j\geq a}H_j+\sum_{b=1}^{a-1}H_{ab}.
\end{equation}
Note that each bargraph $\pi=ab\pi''$, $1\leq b\leq a-1$, can be
written as either $ab^m$ (where we define $b^m$ to be the word
$bb\cdots b$), $ab^mj\pi'$ with $j\geq b+1$, or $ab^mj\pi'$ with
$j\leq b-1$. Thus, for all $1\leq b\leq a-1$, we have
\begin{align*}
H_{ab}&=\sum_{m\geq 1}x^{a+bm}y^{m+1}q_{a-b,m}+\sum_{m\geq 1}\left(x^{a+bm}y^{m+1}q_{a-b,m}\sum_{j\geq b+1}H_{j}\right)\\
&+\sum_{m\geq1}\left(x^{a+b(m-1)}y^mq_{a-b,m}\sum_{c=1}^{b-1}H_{bc}\right),
\end{align*}
which is equivalent to
\begin{align}\label{eqcA3}
H_{ab}&=\sum_{m\geq 1}x^{a+bm}y^{m+1}q_{a-b,m}\\
&+\sum_{m\geq1}\left(x^{a+b(m-1)}y^{m}q_{a-b,m}\left(x^by\sum_{j\geq
b+1}H_{j}+\sum_{c=1}^{b-1}H_{bc}\right)\right).\notag
\end{align}
Thus, by \eqref{eqcA2}, we have that
$H_a-x^ay-x^ayH_a=x^ay\sum_{j\geq a+1}H_j +
\sum_{b=1}^{a-1}H_{ab}$, which, by \eqref{eqcA3}, leads to
\begin{equation}
H_{ab}=\alpha_{ab}(1-x^by)H_b\mbox{ with }\alpha_{ab}=\sum_{m\geq
1} x^{a+b(m-1)}y^{m}q_{a-b,m}.
\end{equation}
Therefore, by \eqref{eqcA1} and \eqref{eqcA2}, we can write
\begin{equation}\label{eqcA4}
H_a=x^ayH+\sum_{b=1}^{a-1}\beta_{ab}H_b
\end{equation}
with $\beta_{ab}=\alpha_{ab}(1-x^by)-x^ay$.
\begin{lemma}\label{lemcA1}
For all $a\geq1$,
\begin{align*}
H_a=H\left(x^ay+\sum_{j=1}^a\left(x^jy\sum_{s\geq0}L_a(j,s)\right)\right),
\end{align*}
where
$L_a(j,s)=\sum_{j=i_{s+1}<i_s<\cdots<i_1<i_0=a}\prod_{\ell=0}^s\beta_{i_\ell
i_{\ell+1}}$.
\end{lemma}
\begin{proof}
We prove it by induction on $a$. For $a=1$, this gives
$H_1=xyH$ as expected (by removing the leftmost column of the
bargraph $1\pi'$). Assume that the claim holds for $1,2,\cdots,a$,
and let us prove it for $a+1$. By \eqref{eqcA4}, we have
$$H_{a+1}=x^{a+1}yH+\sum_{b=1}^{a}\beta_{(a+1)b}H_b.$$
Thus, by induction assumption, we obtain
\begin{align*}
H_{a+1}&=x^{a+1}yH+\sum_{b=1}^{a}\beta_{(a+1)b}H\left(x^by+\sum_{j=1}^bx^jy\left(\sum_{s\geq0}L_b(j,s)\right)\right)\\
&=H\left(x^{a+1}y+\sum_{b=1}^{a}\beta_{(a+1)b}x^by+\sum_{b=1}^{a}\sum_{j=1}^bx^jy\beta_{(a+1)b}\left(\sum_{s\geq0}L_b(j,s)\right)\right)\\
&=H\left(x^{a+1}y+\sum_{b=1}^{a}\beta_{(a+1)b}x^by\right)\\
&\qquad+H\left(\sum_{j=1}^{a}\sum_{b=j}^ax^jy\left(\sum_{s\geq0}\sum_{j=i_{s+1}<i_s<\cdots<i_1<i_0=b<i_{-1}=a+1}\prod_{\ell=-1}^s\beta_{i_\ell i_{\ell+1}}\right)\right)\\
&=H\left(x^{a+1}y+\sum_{b=1}^{a}x^by\sum_{i_1=b<i_0=a+1}\beta_{i_0i_1}\right)\\
&\qquad+H\sum_{j=1}^ax^jy\left(\sum_{s\geq0}\sum_{j=i_{s+1}<i_s<\cdots<i_1<i_0<i_{-1}=a+1}\prod_{\ell=-1}^s\beta_{i_\ell i_{\ell+1}}\right)\\
&=H\left(x^{a+1}y+\sum_{j=1}^{a}x^jy\left(\sum_{s\geq0}L_{a+1}(j,s)\right)\right),
\end{align*}
which completes the proof.
\end{proof}

By \eqref{eqcA1} and Lemma \ref{lemcA1}, we can state our first
main result.

\begin{theorem}\label{mth1}
The generating function $H(x,y,\qq)$ is given by
\begin{equation*}
H(x,y,\qq)=\frac{1}{1-\frac{xy}{1-x}-\sum_{j\geq 1}\left(x^jy
\sum_{s\geq0}L(j,s)\right)},
\end{equation*}
where
$L(j,s)=\sum_{j=i_{s+1}<i_s<\cdots<i_1<i_0}\prod_{\ell=0}^s\beta_{i_\ell
i_{\ell+1}}$.
\end{theorem}

For instance, if $q_{a,b}=1$ for all $a,b\geq 1$, then
$\alpha_{ab}=\sum_{m\geq 1}x^{a+b(m-1)}y^m=\frac{x^ay}{1-x^by}$,
which yields $\beta_{ab}=\alpha_{ab}(1-x^by)-x^ay=0$. Thus, in
this case, Theorem \ref{mth1} shows that
$H(x,y,1,1,\ldots)=\frac{1-x}{1-x-xy}$, as expected.
\vspace{1cc}
\begin{center}
{\bf 2.1. Counting all corners of type A}
\end{center}

Let $q_{a,b}=q$ for all $a,b\geq 1$. From the definitions, we have
$\alpha_{ab}=q\frac{x^ay}{1-x^by}$ and $\beta_{ab}=(q-1)x^ay$. Therefore,
\begin{align*}
L(j,s)&=\sum_{j=i_{s+1}<i_s<\cdots<i_1<i_0}\prod_{\ell=0}^s(q-1)x^{i_{\ell}y}\\
&=\sum_{j=i_{s+1}<i_s<\cdots<i_1<i_0}(q-1)^{s+1}y^{s+1}x^{i_0+i_1+\cdots+i_s}\\
&=(q-1)^{s+1}y^{s+1}\frac{x^{(s+1)j+\binom{s+2}{2}}}{(1-x)(1-x^2)\cdots(1-x^{s+1})}.
\end{align*}
Thus the generating function $F=H(x,y,q,q,\cdots)$ is given by
\begin{align*}
F&=\frac{1}{1-\frac{xy}{1-x}-\displaystyle \sum_{j\geq 1}x^jy\sum_{s\geq 0}\frac{(q-1)^{s+1}y^{s+1}x^{(s+1)j+\binom{s+2}{2}}}{(1-x)(1-x^2)\cdots(1-x^{s+1})}}\\
&=\frac{1}{1-\frac{xy}{1-x}-\displaystyle \sum_{s\geq 0}\frac{(q-1)^{s+1}y^{s+2}x^{(s+2)+\binom{s+2}{2}}}{(1-x)(1-x^2)\cdots(1-x^{s+2})}}\\
&=\frac{1}{1-\frac{xy}{1-x}-\displaystyle \sum_{s\geq 1}\frac{(q-1)^sy^{s+1}x^{\binom{s+2}{2}}}{(1-x)(1-x^2)\cdots(1-x^{s+1})}}.
\end{align*}

Let $\cora(\pi)$ be the number of corners of type A in $\pi$.
We define $g_{n,k}=\sum_{\pi \in \BB_{n,k}}\cora(\pi)$ and $g_n=\sum_{k\geq 1}g_{n,k}$.
Let $G(x,y)=\sum_{n,k\geq 1}g_{n,k}x^ny^k$ be the generating function for the total number of type A corners over all bargraphs according to the number of cells and columns.  Then, it follows that
\begin{align*}
G(x,y)=\frac{\partial F}{\partial
q}\Bigr|_{\substack{q=1}}=\frac{y^2x^3}{(1-x-xy)^2(1+x)}.
\end{align*}
Note that $G(x,1)=\frac{x^3}{(1-2x)(1+x)}$ is the generating function for the total number of type A corners over all bargraphs according to the number of cells. Hence,
\begin{equation}
g_n=\left(\frac{n+1}{12}-\frac{2}{9}\right)2^n-\frac{1}{9}(-1)^n.
\end{equation}

\vspace{1.1cc}
\begin{center}
{\bf 2.2. Counting $(v,w)$-corners of type A}
\end{center}

Fix $v,w\geq1$. Define $q_{v,w}=q$ and $q_{a,b}=1$ for all
$(a,b)\neq(v,w)$. Then we have
\begin{align*}
\alpha_{ab}&=\sum_{m\geq1} x^{a+b(m-1)}y^{m}q_{a-b,m}=\sum_{m\geq1} x^{a+b(m-1)}y^{m}+x^{a+b(w-1)}y^w(q_{a-b,w}-1)\\
&=\frac{x^ay}{1-x^by}+x^{a+b(w-1)}y^w(q-1)\delta_{a-b=v},
\end{align*}
where $\delta_{\chi}=1$ if $\chi$ holds, and $\delta_{\chi}=0$ otherwise. Hence,
\begin{align}\label{eqcA5}
\beta_{ab}&=\alpha_{ab}(1-x^by)-x^ay
=x^{a+b(w-1)}y^w(q-1)\delta_{a-b=v}(1-x^by).
\end{align}
Recall that $L(j,s)=\sum_{j=i_{s+1}<i_s<\cdots<i_1<i_0}\prod_{\ell=0}^s\beta_{i_\ell
i_{\ell+1}}$. By using \eqref{eqcA5}, we have
{\small\begin{align*}
&L(j,s)\\
&=\sum_{j=i_{s+1}<i_s<\cdots<i_1<i_0}\prod_{\ell=0}^s\left(x^{i_{\ell}+i_{\ell+1}(w-1)}y^w(q-1)\delta_{i_{\ell}-i_{\ell+1}=v}(1-x^{i_{\ell+1}}y)\right)\\
&=\sum_{j=i_{s+1}<i_s<\cdots<i_1<i_0}(q-1)^{s+1}y^{w(s+1)}x^{\sum_{\ell=0}^si_{\ell}+(w-1)i_{\ell+1}}\prod_{\ell=0}^s\left(\delta_{i_{\ell}-i_{\ell+1}=v}(1-x^{i_{\ell+1}}y)\right)\\
&=(q-1)^{s+1}y^{w(s+1)}x^{wj(s+1)+v\binom{s+2}{2}+(w-1)v\binom{s+1}{2}}\prod_{\ell=0}^s(1-x^{j+(s-l)v}y).
\end{align*}}
From Theorem~\ref{mth1}, we obtain that the generating function
$F=H(x,y,\qq)$ is given by
{\small\begin{equation*}
F=\frac{1}{1-\frac{xy}{1-x}-\displaystyle \sum_{j\geq
1}x^jy\sum_{s\geq
0}(q-1)^{s+1}y^{w(s+1)}x^{wj(s+1)+v\binom{s+2}{2}+(w-1)v\binom{s+1}{2}}\prod_{\ell=0}^s(1-x^{j+\ell
v}y)}.
\end{equation*}}
Recall that $\Lambda_{(v,w)}(\pi)$ denotes the number of $(v,w)$-corners of type A in $\pi$.
We define $t_{n,k}=\sum_{\pi \in \BB_{n,k}}\Lambda_{(v,w)}(\pi)$ and $t_n=\sum_{k\geq 1}t_{n,k}$.
Let $T(x,y)=\sum_{n,k\geq 1}t_{n,k}x^ny^k$ be the generating function for the total number of $(v,w)$-corners of type A over all bargraphs according to the number of cells and columns.  Then, it follows that
\begin{align*}
T(x,y)=\frac{\partial F}{\partial
q}\Bigr|_{\substack{q=1}}=\frac{x^{v+w+1}y^{w+1}}{(1-\frac{xy}{1-x})^2}\frac{1-xy-x^{w+2}(1-y)}{(1-x^{w+1})(1-x^{w+2})},
\end{align*}
which leads to
$$T(x,1)=\frac{x^{v+w+1}}{(1-2x)^2}\frac{(1-x)^3}{(1-x^{w+1})(1-x^{w+2})};$$
the generating function for the total number of $(v,w)$-corners of
type A over all bargraphs according to the number of cells. As
a consequence, we have the following result.
\begin{corollary}
The total number of $(v,w)$-corners of type A over all bargraphs
having $n$ cells is given by
$t_n=\frac{n}{(2^{w+1}-1)(2^{w+2}-1)}2^{w-v+n-1}$.
\end{corollary}

\vspace{.5cc}
\begin{center}
{\bf 2.3. Restricted bargraphs}
\end{center}

Theorem \ref{mth1} can be refined as follows. Fix $N\geq1$. Let
$H^{(N)}:=H^{(N)}(x,y,\qq)$ be the generating function for the
number of bargraphs $\pi$ such that the height of each column is at most $N$ according to the number of cells in $\pi$, the
number of columns of $\pi$, and the number of $(a,b)$-corners of
type A in $\pi$ corresponding to the variables $x,y$ and
$\qq=(q_{a,b})_{a,b\geq 1}$ respectively. Then by using similar
arguments as in the proof of \eqref{eqcA4}, we obtain
\begin{equation}\label{eqcNA1}
H_a^{(N)}=x^ayH^{(N)}+\sum_{b=1}^{a-1}\beta_{ab}H_b^{(N)},
\end{equation}
where $H^{(N)}_a:=H^{(N)}_a(x,y,\qq)$ is the generating function
for the number of bargraphs $\pi=a\pi'$ such that the height of
each column is a maximum of $N$. Clearly,
$H^{(N)}=1+\sum_{a=1}^NH_a^{(N)}$. By the proof of Theorem
\ref{mth1}, we can state its extension as follows.
\begin{theorem}\label{mthN1}
The generating function $H^{(N)}(x,y,\qq)$ is given by
\begin{equation*}
H^{(N)}(x,y,\qq)=\frac{1}{1-y\sum_{j=1}^Nx^j-\sum_{j=1}^N\left(x^jy
\sum_{s\geq0}L(j,s)\right)},
\end{equation*}
where $$L(j,s)=\sum_{j=i_{s+1}<i_s<\cdots<i_1<i_0\leq
N}\prod_{\ell=0}^s\beta_{i_\ell i_{\ell+1}}.$$ Moreover, for all
$a=1,2,\ldots,N$, we have
\begin{align*}
H_a^{(N)}=H^{(N)}\left(x^ay+\sum_{j=1}^a\left(x^jy\sum_{s\geq0}L_a(j,s)\right)\right),
\end{align*}
where
$L_a(j,s)=\sum_{j=i_{s+1}<i_s<\cdots<i_1<i_0=a}\prod_{\ell=0}^s\beta_{i_\ell
i_{\ell+1}}$.
\end{theorem}

For instance, Theorem \ref{mthN1} for $N=1,2$ gives
$H^{(1)}(x,y,\qq)=\frac{xy}{1-xy}$ and
$$H^{(2)}(x,y,\qq)=\frac{1}{1-(x+x^2)y-x^2(1-xy)\sum_{m\geq1}x^my^mq_{1,m}+x^3y^2}.$$

\vspace{.8cc}
\begin{center}
{\bf 2.4. Counting corners of type A in set partitions}
\end{center}

Recall that we represent any set
partition as a bargraph corresponding to its canonical sequential representation. Let $P_k(x,y,\qq)$ be the generating function for the number of
set partitions $\pi$ of $[n]$ with exactly $k$ blocks according to
the number of cells in $\pi$, the number of columns of $\pi$
(which is $n$), and the number of $(a,b)$-corners of type A in
$\pi$ corresponding to the variables $x,y$ and
$\qq=(q_{a,b})_{a,b\geq 1}$ respectively.

Note that each set partition with exactly $k$ blocks can be
decomposed as $1\pi^{(1)}\cdots k\pi^{(k)}$ such that $\pi^{(j)}$
is a word over alphabet $[j]$. Thus, by Theorem \ref{mthN1}, we
have the following result.
\begin{theorem}\label{mthPk1}
The generating function $P_k(x,y,\qq)$ is given by
\begin{align*}
P_k(x,y,\qq)=\prod_{N=1}^kH^{(N)}_N(x,y,\qq)=\prod_{N=1}^k\frac{x^Ny+\sum_{j=1}^N\left(x^jy\sum_{s\geq0}L_N(j,s)\right)}
{1-y\sum_{j=1}^Nx^j-\sum_{j=1}^N\left(x^jy\sum_{s\geq0}L(j,s)\right)},
\end{align*}
where $$\displaystyle L(j,s)=\sum_{j=i_{s+1}<i_s<\cdots<i_0\leq
N} \prod_{\ell=0}^s\beta_{i_\ell i_{\ell+1}}\mbox{ and }\displaystyle L_N(j,s)=\sum_{j=i_{s+1}<i_s<\cdots<i_0=N} \prod_{\ell=0}^s\beta_{i_\ell i_{\ell+1}}.$$
\end{theorem}

Now, we consider counting all corners of type A in set partitions. Let $q_{a,b}=q$ for all $a,b\geq 1$, and
$Q_k(x,y)=\frac{\partial}{\partial q}P_k(x,y,\qq)\Bigr|_{\substack{q=1}}$.
Note that for any $s\geq 0$ and $1\leq j\leq N-1$,
$$L_N(j,s)=(q-1)^{s+1}y^{s+1}\sum_{j=i_{s+1}<i_s<\cdots<i_0=N}x^{\sum_{\ell=0}^sx_{\ell}}.$$
We have a similar expression for $L(j,s)$.
From Theorem~\ref{mthPk1}, we have the generating function $Q_k(x,y)$ given by
\begin{align*}
&Q_k(x,y)\\
&=\prod_{N=1}^k\frac{x^Ny}{1-y\sum_{j=1}^Nx^j}
\sum_{N=1}^k\frac{\sum_{j=1}^{N-1}\left(x^jy(1-y\sum_{j=1}^Nx^j)+x^jy^2(x^{j+1}+\cdots+x^N)\right)} {1-y\sum_{j=1}^Nx^j}.
\end{align*}
Let $\phi(t)=\frac{t^k}{(1-t)(1-2t)\cdots(1-kt)}$. Then we have $\phi^{\prime}(t)=\frac{t^{k-1}}{(1-t)\cdots(1-kt)}\sum_{j=1}^k\frac{1}{1-jt}$.

Note that
\begin{align*}
Q_k(1,t)&=\phi(t)\sum_{N=2}^k\left((N-1)t+\frac{t^2\sum_{j=1}^{N-1}(N-j)}{1-Nt}\right)\\
&=\phi(t) t\sum_{N=2}^k\left((N-1)+\frac{t\binom{N}{2}}{1-Nt}\right)\\
&=\phi(t)t\left(\binom{k}{2}+\frac{1}{2}\sum_{N=2}^k\frac{tN(N-1)}{1-Nt}\right)\\
&=\phi(t)t\left(\frac{1}{2}\binom{k}{2}+\frac{1}{2}\sum_{N=1}^k\frac{N-1}{1-Nt}\right)\\
&=\frac{1}{2}\binom{k}{2}\phi(t)t+\frac{1}{2}\phi(t)\sum_{N=1}^k\left(-1-t\frac{1}{1-Nt}+\frac{1}{1-Nt}\right)\\
&=\frac{1}{2}\binom{k}{2}t\phi(t)-\frac{1}{2}k\phi(t)+\frac{1}{2}t\phi^{\prime}(t)-\frac{1}{2}t^2\phi^{\prime}(t).
\end{align*}
Let $q_{n,k}$ be the coefficient of $t^n$ in $Q_k(1,t)$. Define
$\tilde{Q}_k(t)=\sum_{n\geq k}q_{n,k}\frac{t^n}{n!}$ to be the exponential generating function for $q_{n,k}$. Recall that the ordinary and exponential generating functions for Stirling numbers of the second kind $S_{n,k}$ are given by $\phi(t)$ and
$\frac{(e^t-1)^k}{k!}$, respectively.

Thus,
\begin{align*}
\tilde{Q}_k(t)&=\frac{1}{2}\binom{k}{2}\int_0^t\frac{(e^r-1)^k}{k!}dr
-\frac{k(e^t-1)^k}{2k!}+\frac{kt(e^t-1)^{k-1}e^t}{2k!}\\
&-\int_0^t\frac{rk(e^r-1)^{k-1}e^r}{2k!}dr.
\end{align*}
Hence, the exponential generating function $\tilde{Q}(t,y)=\sum_{k\geq0}\tilde{Q}_k(t)y^k$ for the total number of corners over set partitions of $[n]$ with $k$ blocks is given by
\begin{align*}
\tilde{Q}(t,y)&=\frac{y^2}{4}\int_0^t(e^r-1)^2e^{y(e^r-1)}dr+\frac{yt}{2}e^{t+y(e^t-1)}\\
&-\frac{y}{2}(e^t-1)e^{y(e^t-1)}-\frac{y}{2}\int_0^tre^{r+y(e^r-1)}dr.
\end{align*}
In particular, we have
\begin{align*}
\frac{\partial}{\partial t}\tilde{Q}(t,y)&=\frac{y^2}{4}(2te^{2t+y(e^t-1)}-e^{2t+y(e^t-1)}+e^{y(e^t-1)})\\
&=\frac{2t-1}{4}\frac{\partial^2}{\partial t^2}e^{y(e^t-1)}-\frac{2t-1}{4}\frac{\partial}{\partial t}e^{y(e^t-1)}+\frac{y^2}{4}e^{y(e^t-1)}.
\end{align*}
Hence, we can state the following result.
\begin{theorem}\label{thmcAsp}
The total number of corners of type A over set partitions of $[n+1]$ with $k$ blocks is given by
$$\frac{n}{2}S_{n+1,k}-\frac{1}{4}S_{n+2,k}-\frac{n}{2}S_{n,k}+\frac{1}{4}S_{n+1,k}+
\frac{1}{4}S_{n,k-2}.$$
Moreover, the total number of corners of type A over set partitions of $[n+1]$ is given by
$$\frac{2n+1}{4}B_{n+1}-\frac{1}{4}B_{n+2}-\frac{2n-1}{4}B_n,$$
where $B_n$ is the $n^{th}$ Bell number.
\end{theorem}
\vspace{1.1cc}
\begin{center}
{\bf 3. Counting Corners of type B}
\end{center}

Let $J:=J(x,y,\pp)$ be the generating function for the number of
bargraphs $\pi$ according to the number of cells in $\pi$, the
number of columns of $\pi$, and the number of $(a,b)$-corners of
type B in $\pi$, corresponding to the variables $x,y$ and
$\pp=(p_{a,b})_{a,b\geq 1}$ respectively, that is,
$$J=\sum_{n\geq 0}\sum_{\pi\in \BB_n}x^ny^{\text{col}(\pi)}\prod_{a,b\geq
1}p_{a,b}^{\Lambda_{(a,b)}(\pi)},$$ where $\Lambda_{(a,b)}(\pi)$
denotes the number of $(a,b)$-corners of type B in $\pi$. From the
definitions, we have
\begin{equation}\label{eqcB1}
J=1+\sum_{a\geq 1}J_a,
\end{equation}
where $J_a$ is the generating function for the number of
bargraphs $\pi=a\pi'$ in which the height of the first column is $a$.
Since each bargraph $\pi=a\pi'$ can be decomposed as either
$\pi=a^m$, $\pi=a^mb\pi''$ with $b\geq a+1$, or $\pi=a^mb\pi''$
with $1\leq b\leq a-1$, we have
\begin{align}\label{eqcB2}
J_a&=\sum_{m\geq 1}x^{am}y^mp_{m,a} + \sum_{m\geq 1}x^{am}y^m(J_{a+1}+J_{a+2}+\cdots) \notag\\
& + \sum_{m\geq 1}\sum_{b=1}^{a-1}x^{am}y^mp_{m,a-b}J_b.
\end{align}
Define $\gamma_a:=\sum_{m\geq 1}x^{am}y^mp_{m,a}$. It follows from \eqref{eqcB1} that $$J-1-\sum_{b=1}^aJ_b=\sum_{b\geq a+1}J_b.$$ Then we
obtain
\begin{align*}
J_a&=\gamma_a+\frac{x^ay}{1-x^ay}\left(J-1-\sum_{b=1}^aJ_b\right)+\sum_{m\geq 1}\left(x^{am}y^m\sum_{b=1}^{a-1}p_{m,a-b}J_b\right)\\
&=\gamma_a+\frac{x^ay}{1-x^ay}(J-1)-\frac{x^ay}{1-x^ay}J_a+\sum_{m\geq
1}\left(x^{am}y^m\sum_{b=1}^{a-1}(p_{m,a-b}-1)J_b\right),
\end{align*}
which, by solving for $J_a$, gives
\begin{align*}
J_a&=x^ay(J-1)+(1-x^ay)\gamma_a+(1-x^ay)\sum_{m\geq
1}\left(x^{am}y^m\sum_{b=1}^{a-1}(p_{m,a-b}-1)J_b\right).
\end{align*}
If we define
\begin{equation*}
\theta_a:=x^ay(J-1)+(1-x^ay)\gamma_a \hbox{ and }\mu_{a,b}:=(1-x^ay)\sum_{m\geq 1}\left(x^{am}y^m(p_{m,a-b}-1)\right),
\end{equation*}
then we obtain
\begin{equation}\label{eqcB3}
J_a=\theta_a+\sum_{b=1}^{a-1}\mu_{a,b}J_b.
\end{equation}
By similar techniques as in the proof of Lemma \ref{lemcA1}, we
can state the following result.
\begin{lemma}\label{lemcB1}
For all $a\geq1$,
\begin{align*}
J_a=\theta_a+\sum_{j=1}^{a-1}\Gamma_{a,j}\theta_j,
\end{align*}
where
$\Gamma_{a,j}=\sum_{s\geq0}\sum_{j=i_{s+1}<i_s<\cdots<i_0=a}\prod_{\ell=0}^s\mu_{i_\ell
i_{\ell+1}}$.
\end{lemma}

\begin{theorem}\label{mth2}The generating function $J(x,y,\pp)$ is given by
\begin{equation*}
J(x,y,\pp)=1+\frac{\sum_{m\geq 1}\sum_{j\geq 1}(1+\Gamma_j)(1-x^jy)x^{jm}y^mp_{m,j}}{1-\frac{xy}{1-x}-\sum_{j\geq 1}x^jy
\Gamma_j},
\end{equation*}
where $\Gamma_j=\sum_{s\geq 0}\sum_{j=i_{s+1}<i_s<\cdots<i_0}\prod_{\ell=0}^s\mu_{i_{\ell}i_{\ell+1}}$.
\end{theorem}

For instance, if $p_{a,b}=1$ for all $a,b\geq 1$, then $\mu_{a,b}=0$ which implies that $\Gamma_a=0$. Thus Theorem~\ref{mth2} shows that $J(x,y,1,1,\cdots)=\frac{1-x}{1-x-xy}$.

\vspace{1cc}
\begin{center}
{\bf 3.1. Counting all corners of type B}
\end{center}

Let $p_{a,b}=p$ for all $a,b\geq 1$. From the definitions, we have $\mu_{a,b}=(p-1)x^ay$ which yields
\begin{align}\label{eqcB4}
\Gamma_j&=\sum_{s\geq 0}\left((p-1)^{s+1}\sum_{j=i_{s+1}<i_s<\cdots<i_0}x^{i_0+\cdots+i_s}\right)\notag\\
&=\sum_{s\geq 0}\frac{(p-1)^{s+1}y^{s+1}x^{(s+1)j+\binom{s+2}{2}}}{(1-x)(1-x^2)\cdots(1-x^{s+1})}.
\end{align}
From Theorem~\ref{mth2} and \eqref{eqcB4}, the generating function $F=J(x,y,p,p,\cdots)$ is given by
\begin{align*}
F&=1+\frac{p\frac{xy}{1-x}+p\sum_{j\geq 1}\Gamma_jx^jy}{1-\frac{xy}{1-x}-\sum_{j\geq 1}\Gamma_jx^jy}
=1+\frac{p\frac{xy}{1-x}+p\sum_{s\geq 0}\frac{(p-1)^{s+1}y^{s+2}x^{\binom{s+3}{2}}}{(1-x)(1-x^2)\cdots(1-x^{s+2})}}{1-\frac{xy}{1-x}-\sum_{s\geq 0}\frac{(p-1)^{s+1}y^{s+2}x^{\binom{s+3}{2}}}{(1-x)(1-x^2)\cdots(1-x^{s+2})}}.
\end{align*}
Let $\corb(\pi)$ be the number of corners of type B in $\pi$.
Define $h_{n,k}=\sum_{\pi \in \BB_{n,k}}\corb(\pi)$ and $h_n=\sum_{k\geq 1}h_{n,k}$.
Let $H(x,y)=\sum_{n,k\geq 1}h_{n,k}x^ny^k$ be the generating function for the total number of type B corners over all bargraphs according to the number of cells and columns.  Then, it follows that
\begin{equation}
H(x,y)=\frac{\partial F}{\partial
p}\Bigr|_{\substack{p=1}}=\frac{xy(1-x-xy+x^2y^2)}{(1-x-xy)^2}.
\end{equation}
Note that $H(x,1)=\frac{x(x-1)^2}{(1-2x)^2}$ is the generating function for the total number of type B corners over all bargraphs according to the number of cells.

\vspace{.5cc}
\begin{center}
{\bf 3.2. Counting $(v,w)$-corners of type B}
\end{center}

Fix $v,w\geq1$. Define $p_{v,w}=p$ and $p_{a,b}=1$ for all $(a,b)\neq(v,w)$. Then we have
$\mu_{a,b}=(1-x^ay)x^{av}y^v(p-1)\delta_{a-b=w}$ which yields
\begin{align}\label{eqcB5}
\Gamma_j&=\sum_{s\geq 0}\sum_{j=i_{s+1}<i_s<\cdots<i_0}\prod_{\ell=0}^s\left((1-x^{i_{\ell}}y)x^{i_{\ell}v}y^v(p-1)\delta_{i_{\ell+1}-i_{\ell}=w}\right)\notag\\
&=\sum_{s\geq 0} (p-1)^{s+1}y^{v(s+1)}x^{vj(s+1)+vw\binom{s+1}{2}}\prod_{\ell=0}^s(1-x^{j+(\ell+1) w}y).
\end{align}
From Theorem~\ref{mth2} and \eqref{eqcB5}, the generating function $F=J(x,y,\pp)$ is given by
\begin{align*}
F=1+\frac{\frac{yx}{1-x}+(1+\Gamma_w)(1-x^wy)x^{wv}y^v(p-1)+y\displaystyle\sum_{j\geq 1}x^j\Gamma_j}{1-\frac{xy}{1-x}-y\displaystyle\sum_{j\geq 1}x^j\Gamma_j}.
\end{align*}

Recall that $\Lambda_{(v,w)}(\pi)$ denotes the number of $(v,w)$-corners of type B in $\pi$.
We define $t_{n,k}=\sum_{\pi \in \BB_{n,k}}\Lambda_{(v,w)}(\pi)$ and $t_n=\sum_{k\geq 1}t_{n,k}$.
Let $T(x,y)=\sum_{n,k\geq 1}t_{n,k}x^ny^k$ be the generating function for the total number of $(v,w)$-corners of type B over all bargraphs according to the number of cells and columns.  Then, it follows that
\begin{align*}
T(x,y)&=\frac{\partial F}{\partial
p}\Bigr|_{\substack{p=1}}\\
&=\frac{(1-x^wy)x^{vw}y^w}{\left(1-\frac{xy}{1-x}\right)^2}+\frac{y^{v+1}x^{2v+3}(1-x^wy)+(yx)^{v+1}(1-x^{w+1}y)}{(1-x^{v+1})(1-x^{v+2})\left(1-\frac{xy}{1-x}\right)^2},
\end{align*}
which leads to
$$T(x,1)=\frac{(1-x^w)x^{vw}}{\left(1-\frac{x}{1-x}\right)^2}+\frac{x^{2v+3}(1-x^w)+x^{v+1}(1-x^{w+1})}{(1-x^{v+1})(1-x^{v+2})\left(1-\frac{x}{1-x}\right)^2};$$
this latter is the generating function for the total number of $(v,w)$-corners of
type B over all bargraphs according to the number of cells.

\vspace{1cc}
\begin{center}
{\bf 3.3. Restricted bargraphs}
\end{center}

Theorem \ref{mth2} can be refined as follows. For $N\geq1$, let
$J^{(N)}:=J^{(N)}(x,y,\pp)$ be the generating function for the
number of bargraphs $\pi$ such that the height of each column is
at most $N$ according to the number of cells in $\pi$, the number
of columns of $\pi$, and the number of $(a,b)$-corners of type B
in $\pi$ corresponding to the variables $x,y$ and
$\pp=(p_{a,b})_{a,b\geq 1}$ respectively. Then by using similar
arguments as in the proof of \eqref{eqcB1} and \eqref{eqcB3}, we
obtain that $$J^{(N)}=1+\sum_{a=1}^NJ_a^{(N)}\mbox{ and }
J_a^{(N)}=\theta_a+\sum_{b=1}^{a-1}\mu_{a,b}J_b^{(N)},$$ for all
$a=1,2,\ldots,N$, where $J^{(N)}_a:=J^{(N)}_a(x,y,\pp)$ is the
generating function for the number of bargraphs $\pi=a\pi'$ such
that the height of each column is at most $N$. From the proof of
Theorem \ref{mth2}, we can state its extension as follows.
\begin{theorem}\label{mthN2}
The generating function $J^{(N)}=J^{(N)}(x,y,\pp)$ is given by
\begin{equation*}
J^{(N)}=1+\frac{\sum_{j=1}^N(1+\Gamma_j)(1-x^jy)\gamma_j}{1-y\sum_{j=1}^Nx^j-\sum_{j=1}^Nx^jy
\Gamma_j},
\end{equation*}
where $$\Gamma_j=\sum_{s\geq 0}\sum_{j=i_{s+1}<i_s<\cdots<i_0\leq N}\prod_{\ell=0}^s\mu_{i_{\ell}i_{\ell+1}}.$$ Moreover, for all $a=1,2,\ldots,N$, we have
\begin{align*}
J_a^{(N)}=\left(x^ay+\sum_{j=1}^{a-1}x^jy\Gamma_{a,j}\right)(J^{(N)}-1)
+(1-x^ay)\gamma_a+\sum_{j=1}^{a-1}\Gamma_{a,j}(1-x^jy)\gamma_j,
\end{align*}
where
$\Gamma_{N,j}=\sum_{s\geq0}\sum_{j=i_{s+1}<i_s<\cdots<i_0=N}\prod_{\ell=0}^s\mu_{i_\ell
i_{\ell+1}}$.
\end{theorem}

For instance, Theorem \ref{mthN2} for $N=1$ gives
$$J^{(1)}(x,y,\pp)=1+\frac{(1-xy)\gamma_1}{1-xy}=1+\sum_{m\geq1}x^my^mp_{m,1}.$$
\vspace{.5cc}
\begin{center}
{\bf 3.4. Counting corners of type B in set partitions}
\end{center}

Recall that we represent any set
partition as a bargraph corresponding to its canonical sequential representation. Let $P_k(x,y,\pp)$ be the generating function for the number of
set partitions $\pi$ of $[n]$ with exactly $k$ blocks according to
the number of cells in $\pi$, the number of columns of $\pi$
(which is $n$), and the number of $(a,b)$-corners of type B in
$\pi$ corresponding to the variables $x,y$ and
$\pp=(p_{a,b})_{a,b\geq 1}$ respectively.

Note that each set partition with exactly $k$ blocks can be
decomposed as $1\pi^{(1)}\cdots k\pi^{(k)}$ such that $\pi^{(j)}$
is a word over alphabet $[j]$. Thus, by Theorem \ref{mthN2}, we
have the following result.
\begin{theorem}\label{mthPk1}
Let $p_{a,b}=p$ for all $a,b\geq 1$. Then the generating function $P_k(x,y,\pp)$ is given by
\begin{align*}
P_k(x,y,\pp)=p^{1-k}\prod_{N=1}^kJ^{(N)}_N(x,y,\pp),
\end{align*}
where $J_N^{(N)}$ is given in statement Theorem \ref{mthN2}.
\end{theorem}

Now, we consider counting all corners of type B in set partitions. Let $p_{a,b}=p$ for all $a,b\geq 1$, and
$Q_k(x,y)=\frac{\partial}{\partial p}P_k(x,y,\pp)\Bigr|_{\substack{p=1}}$. By Theorem \ref{mthN2}, we have that $J^{(N)}(x,y,{\bf 1})=\frac{1}{1-y\sum_{j=1}^Nx^j}$ and
$$\frac{\partial }{\partial p}J^{(N)}(x,y,\pp)\mid_{p=1}=
\frac{y\sum_{j=1}^Nx^j-\left(y\sum_{j=1}^Nx^j\right)^2+y^2\sum_{j=1}^Nx^j\frac{x^{j+1}-x^{N+1}}{1-x}}
{\left(1-y\sum_{j=1}^Nx^j\right)^2}.$$
Moreover, Theorem \ref{mthN2} gives that
$J_N^{(N)}(x,y,{\bf 1})=\frac{x^Ny}{1-y\sum_{j=1}^Nx^j}$, and
$$\frac{\partial }{\partial p}J_N^{(N)}(x,y,\pp)\mid_{p=1}=x^Ny\left(\frac{\partial }{\partial p}J^{(N)}(x,y,\pp)\mid_{p=1}
+\frac{1-x^Ny}{1-\sum_{j=1}^Nx^jy}\right).$$
Hence, by Theorem \ref{mthPk1}, we have
$$Q_k(x,y)=\prod_{N=1}^k\frac{x^Ny}{1-y\sum_{j=1}^Nx^j}\left(\sum_{N=1}^k\frac{\frac{\partial }{\partial p}J_N^{(N)}(x,y,\pp)\mid_{p=1}}{\frac{x^Ny}{1-y\sum_{j=1}^Nx^j}}-k+1\right).$$
In particular, the generating function for the total number of corners of type B over all set partitions of $[n]$ with $k$ blocks is given by
$$Q_k(1,t)=\frac{t^k}{(1-t)(1-2t)\cdots(1-kt)}
\left(\sum_{N=1}^k\frac{\frac{\partial}{\partial p}J_N^{(N)}(1,t,\pp)\mid_{p=1}}{\frac{t}{1-Nt}}-k+1\right),$$
which, by $\frac{\partial}{\partial p}J_N^{(N)}(1,t,\pp)\mid_{p=1}=t\left(\frac{Nt}{1-Nt}+\frac{t^2N(N-1)}{2(1-Nt)^2}+\frac{1-t}{1-Nt}\right)$, is equivalent to
\begin{align*}
&Q_k(1,t)\\
&=\frac{t^k}{(1-t)(1-2t)\cdots(1-kt)}
\left(1-k+\sum_{N=1}^k\left(1+(N-1)t+\frac{t^2N(N-1)}{2(1-Nt)}\right)\right).
\end{align*}
Hence,
$$Q_k(1,t)=\frac{t^k}{(1-t)(1-2t)\cdots(1-kt)}
\left(1+\frac{t}{2}\binom{k}{2}+\frac{t}{2}\sum_{N=1}^k\frac{N-1}{1-Nt}\right).$$
Define $\tilde{Q}_k(t)$ to be the corresponding exponential generating function to $Q_k(1,t)$, that is $\tilde{Q}_k(t)=\sum_{n\geq0}q_{n,k}\frac{t^n}{n!}$ where $q_{n,k}$ is the coefficient of $t^n$ in $Q_k(1,t)$. Similarly to Section 2.4, we have
\begin{align*}
\tilde{Q}_k(t)&=\frac{(e^t-1)^k}{k!}+\binom{k}{2}\int_0^t\frac{(e^r-1)^k}{2k!}dr
+\frac{kt(e^t-1)^{k-1}e^t}{2k!}-\frac{k(e^t-1)^k}{2k!}\\
&-\int_0^t\frac{rk(e^r-1)^{k-1}e^r}{2k!}dt.
\end{align*}
Define $\tilde{Q}(t,y)=\sum_{k\geq0}\tilde{Q}_k(t)y^k$; thus by multiplying by $y^k$ and summing over $k\geq1$, we obtain
\begin{align*}
\tilde{Q}(t,y)&=e^{y(e^t-1)}-1+\frac{y^2}{4}\int_0^t(e^r-1)^2e^{y(e^r-1)}dr\\
&+\frac{ty}{2}e^{t+y(e^t-1)}-\frac{y}{2}(e^t-1)e^{y(e^t-1)}-\frac{y}{2}\int_0^tre^{r+y(e^r-1)}dr.
\end{align*}
In particular,
$$\frac{\partial}{\partial t}\tilde{Q}(t,y)=
\frac{y}{4}\left(y(2t-1)e^{2t+y(e^t-1)}+ye^{y(e^t-1)}+4e^{t+y(e^t-1)}\right),$$
which is equivalent to
$$\frac{\partial }{\partial t}\tilde{Q}(t,y)=
\frac{2t-1}{4}\frac{\partial^2}{\partial t^2}e^{y(e^t-1)}-\frac{2t-5}{4}\frac{\partial}{\partial t}e^{y(e^t-1)}+\frac{y^2}{4}e^{y(e^t-1)}.$$
Since $e^{y(e^t-1)}=\sum_{n\geq0}\sum_{k=0}^nS_{n,k}\frac{x^ny^k}{n!}$, $S_{n,k}$ is the Stirling number of the second kind, we obtain the following result.
\begin{theorem}\label{thmcBsp}
The total number of corners of type B over set partitions of $[n+1]$ with $k$ blocks is given by
$$\frac{n}{2}S_{n+1,k}-\frac{1}{4}S_{n+2,k}-\frac{n}{2}S_{n,k}+\frac{5}{4}S_{n+1,k}+\frac{1}{4}S_{n,k-2}.$$
Moreover, the total number of corners of type B over set partitions of $[n+1]$ is given by
$$\frac{2n+5}{4}B_{n+1}-\frac{1}{4}B_{n+2}-\frac{2n-1}{4}B_n,$$
where $B_n$ is the $n^{th}$ Bell number.
\end{theorem}

\vspace{1cc}
{\bf Acknowledgment:} G. Y\i ld\i r\i m would like to thank the Department of Mathematics at the University of Haifa for their warm hospitality during the writing of this paper. The authors posted this version with minor corrections in the statements of Theorem 6 and 11 after the reviewer, David Callan's comment, in MathSciNet regarding Theorem 6.


\vspace{1cc}

\noindent Toufik Mansour, Department of Mathematics, University of Haifa, 3498838 Haifa, Israel, {\it tmansour@univ.haifa.ac.il}
\medskip

\noindent G\"{o}khan Y\i ld\i r\i m, Department of Mathematics, Bilkent University, 06800 Ankara, Turkey, {\it gokhan.yildirim@bilkent.edu.tr}

{\small
\noindent

}
\vspace{2cc}
\begin{thebibliography}{50}
\bibitem{Bl1}
{\small {\sc A. Blecher, C. Brennan, A. Knopfmacher:} {\it Levels in
bargraphs.} Ars Math. Contemp. {\bf 9} (2015), 287--300.}

\bibitem{Bl2}
{\small {\sc A. Blecher, C. Brennan, A. Knopfmacher:} {\it Combinatorial
parameters in bargraphs.} Quaest. Math. {\bf39} (2016), 619--635.}

\bibitem{Bl3}
{\small {\sc A. Blecher, C. Brennan, A. Knopfmacher:} {\it Peaks in
bargraphs.} Trans. Royal Soc. S. Afr. {\bf71} (2016), 97--103.}

\bibitem{Bl5}
{\small {\sc A. Blecher, C. Brennan, A. Knopfmacher, T. Mansour:} {\it Counting corners in partitions.} Ramanujan J. {\bf39:1} (2016), 201--224.}

\bibitem{Bl4}
{\small {\sc A. Blecher, C. Brennan, A. Knopfmacher:} {\it Walls in bargraphs.} Online J. Anal. Comb. {\bf12} (2017), 619--635.}

\bibitem{BMR}
{\small {\sc M. Bousquet-M\'elou, A. Rechnitzer:} {\it The site-perimeter of bargraphs.} Adv. in Appl. Math. {\bf31} (2003), 86--112.}

\bibitem{DE}
{\small {\sc E. Deutsch, S. Elizalde:} {\it Statistics on bargraphs viewed as cornerless Motzkin paths.} Discrete Appl. Math. {\bf221} (2017), 54--66.}

\bibitem{Fer}
{\small {\sc S. Fereti\'{c}:} {\it A perimeter enumeration of column-convex polyominoes.} Discrete Math. Theor. Comput. Sci. {\bf9} (2007), 57--84.}

\bibitem{G}
{\small {\sc A. Geraschenko:} {\it An investigation of skyline polynomials.} Preprint.}

\bibitem{Man}
{\small {\sc T. Mansour:} {\it Interior vertices in set partitions.} Adv. in Appl. Math. {\bf 101} (2018), 60--69.}

\bibitem{MShaSha}
{\small {\sc T. Mansour, A. Shabani, M. Shattuck:} {\it Counting corners in compositions and set partitions presented as bargraphs.} J. of Diff. Eqs. and Appl.  {\bf 24:6} (2018), 992--1015.}

\bibitem{SP}
{\small {\sc J. Osborn, T. Prellberg:} {\it Forcing adsorption of a tethered polymer by pulling.} J. Stat. Mech. (2010), P09018.}

\bibitem{AP}
{\small {\sc A. Owczarek, T. Prellberg:} {\it Exact solution of the discrete $(1+1)$-dimensional SOS model with field and surface interactions.} J. Stat. Phys. {\bf70:5/6} (1993), 1175--1194.}

\bibitem{PB}
{\small {\sc T. Prellberg, R. Brak:} {\it Critical exponents from nonlinear functional equations for partially directed cluster models.} J. Stat. Phys. {\bf78} (1995), 701--730.}

\bibitem{JR}
{\small {\sc E.J. Janse van Rensburg:} {\it The Statistical Mechanics of Interacting Walks, Polygons, Animals and Vesicles.} Oxford University Press, Oxford, 2000.}

\bibitem{Slo}
{\small {\sc N.J. Sloane:} {\it The On-Line Encyclopedia of Integer Sequences.} http://oeis.org, 2010.}

\end{thebibliography}
\end{document}